\documentclass[11pt]{article}
\usepackage{mathrsfs}
\usepackage{amsmath}
\usepackage{amssymb}
\usepackage{graphicx}
\usepackage{epic}
\usepackage{color}

\usepackage{pst-poly}  
\usepackage{pst-plot}  
\usepackage{pst-poly}  

\renewcommand{\paragraph}{\roman{paragraph}}

\usepackage{bm}
\usepackage{hyperref}
\usepackage{tikz}
\usetikzlibrary{arrows,shapes,positioning}
\usetikzlibrary{decorations.markings}
\tikzstyle arrowstyle=[scale=1]
\tikzstyle directed=[postaction={decorate,decoration={markings, mark=at position .65 with {\arrow[arrowstyle]{stealth}}}}]
\tikzstyle reverse directed=[postaction={decorate,decoration={markings, mark=at position .65 with {\arrowreversed[arrowstyle]{stealth};}}}]

\newtheorem{theorem}{Theorem}[section]
\newtheorem{corollary}[theorem]{Corollary}

\newtheorem{problem}[theorem]{Problem}
\newtheorem{lemma}[theorem]{Lemma}
\newtheorem{example}[theorem]{Example}

\newenvironment{proof}{\noindent {\bf Proof.}}{\rule{3mm}{3mm}\par\medskip}

\begin{document}

\title{Changes in the Seidel energy of blow-up graphs under edge deletion \thanks{Supported by The National Natural Science Foundation of China (No. 12571360, 12331012), Excellent University Research and Innovation Team in Anhui Province (No. 2024AH010002, 2025AHGXZK10041)}}

\author{Yayang Liu, Yi Wang\thanks{Corresponding author: wangy@ahu.edu.cn} \\
School of Mathematical Sciences, Anhui University, Hefei 230601,\\ Anhui, China}

\date{}
\maketitle

\begin{abstract}
Let $S(G)$ denote the Seidel matrix of a simple graph $G$, and let 
$E_S(G)$ be the Seidel energy of $G$, defined as the sum of the 
absolute values of the eigenvalues of $S(G)$. In this paper, we study 
the change of Seidel energy under edge deletion. 
For an independent-set blow-up graph 
$G=H[n_1,\ldots,n_p]$, we establish a general structural criterion within the framework of independent-set blow-up graphs. 
More precisely, if the endpoints of the deleted edge $e$ belong to 
blow-up parts of sizes $n_a$ and $n_b$, respectively, then $E_S(G-e)>E_S(G)$ whenever both $n_a,n_b$ are at least $4$, or one is $3$ and the other is at least $6$, or one is $2$ and the other is at least $15$. 

As applications, we obtain the following consequences. First, for every 
Tur\'an graph $T(n,r)$ with $r\geq4$ and $n\geq4r$, deleting any edge 
strictly increases the Seidel energy. Second, for complete multipartite 
graphs, we derive an exact reduced-order spectral criterion for the 
remaining cases not covered by the structural result. This criterion 
determines whether the Seidel energy increases, decreases, or remains 
unchanged after deleting an edge, by using matrices whose orders depend 
only on the number of partite sets. These results provide affirmative answers to two problems proposed by 
Tian et al. 
[\textit{Linear and Multilinear Algebra} 70 (19) (2022), 4597--4614].

\end{abstract}



\section{Introduction}
Throughout this paper, all graphs are simple and finite. Let $G$ be a graph of order $n$, with vertex set $V(G)$ and edge set $E(G)$, and write $e = uv$ for an edge joining distinct vertices $u,v \in V(G)$. A complete $r$-partite graph $K_{p_1, p_2, \dots, p_r}$ is a simple graph whose vertex set can be partitioned into $r$ nonempty parts of sizes $p_1, p_2, \dots, p_r$, respectively, such that two vertices are adjacent if and only if they belong to different parts. A complete $r$-partite graph of order $n$ is called the $r$-partite Tur\'an graph, denoted by $T(n,r)$, if the size of each part is either $\left\lfloor \frac{n}{r} \right\rfloor$ or $\left\lceil \frac{n}{r} \right\rceil$. The adjacency matrix $A(G)$ of $G$ is an $n \times n$ matrix with entries $A(G)_{uv} = 1$ if $uv \in E(G)$ and $0$ otherwise. The Seidel matrix $S(G)$ of $G$ is defined as $S(G) = J - I - 2A(G)$, where $J$ is the all-ones matrix and $I$ is the identity matrix. Equivalently, the diagonal entries of $S(G)$ are $0$, and for $u \neq v$, the $(u,v)$-entry is $-1$ if $uv \in E(G)$ and $1$ otherwise. The eigenvalues of $S(G)$, denoted by $\lambda_1(S(G)) \geq \lambda_2(S(G)) \geq \cdots \geq \lambda_n(S(G))$, are called the Seidel eigenvalues of $G$; they are all real because $S(G)$ is real symmetric. The Seidel energy of $G$ is defined as $E_S(G) = \sum_{i=1}^n |\lambda_i(S(G))|$.

Seidel matrices first appeared in \cite{VS} in the study of equiangular line systems in Euclidean spaces. Haemers \cite{H} introduced the Seidel energy of a graph and conjectured that for every graph $G$ of order $n$, $E_S(G)\ge 2n-2$, which was verified for $n\le 12$ in the same paper. Ghorbani \cite{G} proved it for all graphs $G$ of order $n$ satisfying $|\det S(G)|\ge n-1$. Rizzolo \cite{R} studied determinants of Seidel matrices in connection with this conjecture. Oboudi \cite{O1} proved the conjecture for every $k$-regular graph $G$ of order $n$ with $k\ne \frac{n-1}{2}$ and without any eigenvalue in the interval $(-1,0)$. The conjecture was finally solved by Akbari et al. \cite{AEKN}. Subsequently, Einollahzadeh et al. \cite{EN2} gave a short matrix proof. Further results on Seidel matrices and Seidel energy can be found in \cite{AAD,BH,EN1,GS,O2}.

In recent years, changes in graph energy under edge modifications have attracted considerable attention. Related problems have been investigated for adjacency energy, distance energy, positive $p$-energy, and other types of graph energies. For more details about the background and development, see \cite{AGO,DS1,DS2,LSG,S,SP,SD,TLW,TLC1,VSV} and the references therein. In this paper, we focus on the change of the Seidel energy of a graph under edge deletion. In 2022, Tian et al. \cite{TLC2} observed that the complete graph $K_n$, the empty graph $nK_1$, and the complete bipartite graph $K_{n_1,n_2}$, where $n_1+n_2=n$, all attain the minimum possible Seidel energy among graphs of order $n$. Consequently, for $K_n$ and $K_{n_1,n_2}$, deleting an edge cannot decrease the Seidel energy. Motivated by this fact, they proposed the following problem.

\begin{problem}[\cite{TLC2}]\label{prob1.1}
For a complete $r$-partite graph $K_{p_1, p_2, \dots, p_r}$ with $r \geq 3$ and $p_i \geq 2$, how does the Seidel energy change when any edge is deleted?
\end{problem}
The examples $K_{2,2,3}$ and $K_{3,3,4}$ given by Tian et al. \cite{TLC2} show that deleting an edge may either increase or decrease the Seidel energy. This indicates that Problem~\ref{prob1.1} is more delicate than the complete and complete bipartite cases, and a complete solution seems to require a more refined method. For this reason, Tian et al. considered a special class of complete multipartite graphs, namely Tur\'an graphs. They proved that for the tripartite Tur\'an graph $T(n,3)$ with $n\ge 9$, deleting any edge always increases the Seidel energy. This observation led them to pose the following problem.

\begin{problem}[\cite{TLC2}]\label{prob1.2}
  For any edge $e$ of the Tur\'an graph $T(n, r)$ with $r \ge 4$, does there exist a sufficiently large integer $n_0$ such that $E_S(T(n, r)) < E_S(T(n, r) - e)$ whenever $n \ge n_0$?
\end{problem}

Liu and Chen \cite{LC} verified Problem~\ref{prob1.2} for the case $r=5$. The aim of this paper is to give a unified treatment of Problems \ref{prob1.1} and \ref{prob1.2}  within the framework of independent-set blow-up graphs.

If $H$ is a graph with vertex set $\{1, \ldots, p\}$ and $n_1, \ldots, n_p$ are positive integers, then the independent-set blow-up graph $H[n_1, \ldots, n_p]$ is obtained by replacing each vertex $i$ of $H$ with an independent set of size $n_i$, and replacing each edge $ij$ of $H$ by a complete bipartite graph $K_{n_i, n_j}$. We establish a structural criterion giving sufficient conditions under which deleting an edge strictly increases the Seidel energy of an
independent-set blow-up graph. As an application, we obtain an affirmative answer to Problem~\ref{prob1.2}, with $n_0=4r$. For complete multipartite graphs, we further derive an exact reduced-order spectral criterion that determines the sign of the Seidel energy change in the cases not covered by the structural criterion. Together, these results provide a unified treatment of Problems \ref{prob1.1} and \ref{prob1.2}.

The remainder of the paper is organized as follows. In Section 2, we introduce notation and present several preliminary lemmas. In Section 3, we establish the structural criterion for independent-set blow-up graphs. In Section 4, we derive the reduced-order spectral criterion for complete multipartite graphs. In Section 5, we interpret the Seidel energy as the adjacency energy of a signed complete graph and propose two open problems for further discussion.

\section{Notation and Preliminaries}
In this section, we introduce notation and present several preliminary lemmas that will be used throughout the paper.

\textbf{Notation.} For an $n \times n$ real symmetric matrix $A$, we denote its eigenvalues by $\lambda_1(A)\ge \lambda_2(A)\ge \cdots \ge \lambda_n(A)$. The nuclear norm of $A$, also called the Schatten $1$-norm, is denoted by $\|A\|_*$. Because $A$ is real symmetric, $\|A\|_* = \sum_{i=1}^n |\lambda_i(A)|$. For a vector $x$, we write $\|x\|$ for the Euclidean norm. For two matrices $A$ and $B$, the notation $A \oplus B$ means their block diagonal direct sum. We write $ A \sim B $ if $A$ and $B$ are orthogonally similar. Define $s_-(A) = \sum_{\lambda_i(A) < 0} |\lambda_i(A)|$; that is, $s_-(A)$ is the sum of the absolute values of the negative eigenvalues of $A$.

\begin{lemma}\label{lem2.1}
Let $A$ be a real symmetric matrix. Then $\|A\|_* = tr(A)+2s_-(A)$.
\end{lemma}

\begin{proof}
Since
\[
\operatorname{tr}(A) = \sum_{\lambda_i(A) > 0} \lambda_i(A) - \sum_{\lambda_i(A) < 0} |\lambda_i(A)|,
\]
it follows that
\[
\operatorname{tr}(A) + 2s_-(A) = \sum_{\lambda_i(A) > 0} \lambda_i(A) + \sum_{\lambda_i(A) < 0} |\lambda_i(A)| = \|A\|_*.
\]
\end{proof}
The following basic properties of the nuclear norm will also be used later. While these properties may have appeared in prior literatures, we include their proofs in this section for the sake of a self-contained argument.

\begin{lemma}[\cite{HJ}, Lemma 5.1.2]\label{lem2.2}
For any two matrices $A$ and $B$ of the same size, $\|A+B\|_* \leq \|A\|_* + \|B\|_*$ and $\bigl|\|A\|_* - \|B\|_*\bigr| \leq \|A-B\|_*$. In particular, $\|A\|_* \geq \|B\|_* - \|A-B\|_*$.
\end{lemma}

\begin{lemma}\label{lem2.3}
Let $A=(A_{ij})_{i,j=1}^{r}$ be a real symmetric matrix partitioned into blocks, with diagonal blocks $A_{ii}$ of sizes $m_i\times m_i$. Then
$
\|A\|_{*}\ge \sum_{i=1}^{r}\|A_{ii}\|_{*}.
$
\end{lemma}

\begin{proof}
For each $\varepsilon=(\varepsilon_1,\ldots,\varepsilon_r)\in\{\pm1\}^r$, define $D_\varepsilon=\operatorname{diag}(\varepsilon_1 I_{m_1},\ldots,\varepsilon_r I_{m_r}),$
where $I_{m_i}$ is the identity matrix of size $m_i$. Each $D_\varepsilon$ is orthogonal. Moreover,
$\frac{1}{2^r}\sum_{\varepsilon\in\{\pm1\}^r}D_\varepsilon AD_\varepsilon=\operatorname{diag}(A_{11},\ldots,A_{rr}),$ since the diagonal blocks are preserved while every off-diagonal block is cancelled in the average. By Lemma \ref{lem2.2}, we obtain
\[
\sum_{i=1}^r\|A_{ii}\|_*
=
\|\operatorname{diag}(A_{11},\ldots,A_{rr})\|_*
=
\left\|
\frac{1}{2^r}\sum_{\varepsilon\in\{\pm1\}^r}D_\varepsilon AD_\varepsilon
\right\|_*
\le
\frac{1}{2^r}\sum_{\varepsilon\in\{\pm1\}^r}\|D_\varepsilon AD_\varepsilon\|_*
=
\|A\|_*.
\]
\end{proof}

\begin{lemma}\label{lem2.4}
Let $M = \begin{pmatrix} B & X \\ X^\top & C \end{pmatrix}$ be a real symmetric block matrix, where $B$ is an $m \times m$ matrix, $C$ is a $k \times k$ matrix, and $X$ is an $m \times k$ matrix. Then
$
\|M\|_* \ge \|B\|_* + \|C\|_*.
$
Moreover, the inequality is strict if $B$ has a positive eigenvalue $\lambda > 0$ with unit eigenvector $u \in \mathbb{R}^m$ and $C$ has a negative eigenvalue $\mu < 0$ with unit eigenvector $v \in \mathbb{R}^k$ such that $u^\top X v \neq 0$. The same conclusion holds when $\lambda<0<\mu$.
\end{lemma}

\begin{proof}
The non-strict inequality follows immediately from Lemma \ref{lem2.3} by applying it to the two diagonal blocks $B$ and $C$.

We now prove the strict assertion. Extend $u$ to an orthonormal basis of $\mathbb{R}^m$ consisting of eigenvectors of $B$, with $u$ as the first vector, and let $U_B$ be the orthogonal matrix formed by these basis vectors. Similarly, extend $v$ to an orthonormal basis of $\mathbb{R}^k$ consisting of eigenvectors of $C$, with $v$ as the first vector, and let $U_C$ be the corresponding orthogonal matrix. Set $U = \begin{pmatrix} U_B & 0 \\ 0 & U_C \end{pmatrix}$, which is orthogonal.
Then $\|U^\top M U\|_* = \|M\|_*$, and
\[
U^\top M U =
\begin{pmatrix}
U_B^\top B U_B & U_B^\top X U_C \\
U_C^\top X^\top U_B & U_C^\top C U_C
\end{pmatrix}.
\]
Since $U_B$ diagonalizes $B$, we have $U_B^\top B U_B = \operatorname{diag}(\lambda, \lambda', \dots)$, where $\lambda$ is the chosen positive eigenvalue and $\lambda',\dots$ are the remaining eigenvalues of $B$. Similarly, $U_C^\top C U_C = \operatorname{diag}(\mu, \mu', \dots)$, where $\mu$ is the chosen negative eigenvalue and $\mu',\dots$ are the remaining eigenvalues of $C$. The $(1,1)$-entry of the off-diagonal block $U_B^\top X U_C$ is $\gamma := u^\top X v \neq 0$. Let $P$ be a permutation matrix that places the first coordinate of the $C$-block immediately after the first coordinate of the $B$-block. Since both $U$ and $P$ are orthogonal, $\|P^\top U^\top MUP\|_* = \|M\|_*.$ Partition $P^\top U^\top MUP$ so that one diagonal block is
$
\begin{pmatrix}
\lambda & \gamma \\
\gamma & \mu
\end{pmatrix},
$
and each of the remaining coordinates forms a separate $1\times 1$ diagonal block. By Lemma \ref{lem2.3}, we have
\[
\|M\|_* = \|P^\top U^\top MUP\|_* \ge
\left\| \begin{pmatrix} \lambda & \gamma \\ \gamma & \mu \end{pmatrix} \right\|_*
+ \sum |\lambda'| + \sum |\mu'|,
\]
where the sums run over all remaining eigenvalues of $B$ and $C$, respectively. Since $\sum |\lambda'| = \|B\|_* - |\lambda|$ and $\sum |\mu'| = \|C\|_* - |\mu|$, we obtain
\[
\|M\|_* \ge \left\| \begin{pmatrix} \lambda & \gamma \\ \gamma & \mu \end{pmatrix} \right\|_*
+ \|B\|_* + \|C\|_* - |\lambda| - |\mu|.
\]
Because $\lambda > 0$, $\mu < 0$, and $\gamma \neq 0$, the $2 \times 2$ matrix has determinant $\lambda\mu - \gamma^2 < 0$, so its two eigenvalues have opposite signs. Its nuclear norm is therefore
\[
\left\| \begin{pmatrix} \lambda & \gamma \\ \gamma & \mu \end{pmatrix} \right\|_* = \sqrt{(\lambda - \mu)^2 + 4\gamma^2}.
\]
Since $\gamma \neq 0$, we have $\sqrt{(\lambda - \mu)^2 + 4\gamma^2} > \lambda - \mu = |\lambda| + |\mu|$. Hence
\[
\left\| \begin{pmatrix} \lambda & \gamma \\ \gamma & \mu \end{pmatrix} \right\|_* > |\lambda| + |\mu|.
\]
Substituting this into the inequality gives
\[
\|M\|_* > (|\lambda| + |\mu|) + \|B\|_* + \|C\|_* - |\lambda| - |\mu| = \|B\|_* + \|C\|_*.
\]
The case $\lambda<0<\mu$ is analogous. This completes the proof of the strict inequality.
\end{proof}

Finally, we recall a known formula for the Seidel energy of complete multipartite graphs.

\begin{theorem}[\cite{O2}, Theorem 5]\label{thm2.5}
Let $G = K_{p_1,\ldots,p_r}$ with $r \ge 2$, and let $n = p_1 + \cdots + p_r$. Then $E_S(G) = 2n - 2r - 2\lambda_n(S(G))$, where $\lambda_n(S(G))$ is the smallest Seidel eigenvalue of $G$.
\end{theorem}

\section{A structural criterion for independent-set blow-up graphs}
In this section, we establish a sufficient structural criterion under which the Seidel energy of an independent-set blow-up graph strictly increases after an edge is deleted.

Let $G=H[n_1,\ldots,n_p]$ be an independent-set blow-up graph as defined in Section $1$, with blow-up parts $X_1,\ldots,X_p$, where $|X_i|=n_i$ for each $i\in\{1,\ldots,p\}$. For $i \neq j$, define $\sigma_{ij} = -1$ if $ij \in E(H)$, and $\sigma_{ij} = 1$ otherwise.

We first present a useful orthogonal reduction for the Seidel matrix of an independent-set blow-up graph.
\begin{lemma}\label{lem3.1}
Let $G = H[n_1, \ldots, n_p]$ and set $n = n_1 + \cdots + n_p$. Then $S(G) \sim -I_{n-p} \oplus R_H$, where $R_H = (r_{ij})$ is the $p \times p$ real symmetric matrix defined by $r_{ii} = n_i - 1$ and $r_{ij} = \sigma_{ij} \sqrt{n_i n_j}$ for $i \neq j$. Consequently, $E_S(G) = n - p + \|R_H\|_*$.
\end{lemma}
\begin{proof}
Order the vertices of $G$ so that the vertices of $X_1$ come first, then those of $X_2$, and so on. Under this ordering, every vector in $\mathbb{R}^n$ can be written as $(x_1, \dots, x_p)^\top$ with $x_i \in \mathbb{R}^{n_i}$ corresponding to the vertices of $X_i$. We first construct an orthogonal matrix $Q$. For each $i$, let $\mathbf{1}_i \in \mathbb{R}^{n_i}$ be the all-ones vector, and let $F_i$ be an $n_i \times (n_i - 1)$ matrix whose columns form an orthonormal basis of $\{ x \in \mathbb{R}^{n_i} : \mathbf{1}_i^\top x = 0 \}$. Thus $F_i^\top F_i = I_{n_i-1}$, $F_i^\top \mathbf{1}_i = 0$. If $n_i = 1$, then $F_i$ is understood to be an empty matrix. Let $F = \operatorname{diag}(F_1, \dots, F_p)$. Then $F$ has $n - p$ columns and $F^\top F = I_{n-p}$. For $i = 1, \ldots, p$, define $g_i = \frac{1}{\sqrt{n_i}} \mathbf{1}_{X_i} \in \mathbb{R}^n$, where $\mathbf{1}_{X_i}$ is the indicator vector of $X_i$. Put $G_0 = (g_1, \ldots, g_p)$. Then $G_0^\top G_0 = I_p$, $F^\top G_0 = 0$. Hence $Q = (F \ G_0)$ is an orthogonal matrix.

We now compute $Q^\top S(G)Q$. Let $f$ be a column of $F$ supported on $X_i$. Since the sum of the entries of $f$ on $X_i$ is zero, we have $S(G)f = -f$. Therefore $S(G)F = -F$. It follows that $F^\top S(G)F = -I_{n-p}$ and $F^\top S(G)G_0 = 0$. It remains to compute $G_0^\top S(G) G_0$. For $i = j$,
$$
g_i^\top S(G) g_i = \frac{1}{n_i} \mathbf{1}_{X_i}^\top S(G) \mathbf{1}_{X_i} = \frac{1}{n_i} \mathbf{1}_i^\top (J_{n_i} - I_{n_i}) \mathbf{1}_i = n_i - 1.
$$
For $i \neq j$, the Seidel block between $X_i$ and $X_j$ is $\sigma_{ij} J_{n_i \times n_j}$. Thus
$$
g_i^\top S(G) g_j = \frac{1}{\sqrt{n_i n_j}} \mathbf{1}_{i}^\top \sigma_{ij} J_{n_i \times n_j} \mathbf{1}_{j} = \sigma_{ij} \sqrt{n_i n_j}.
$$
Hence $G_0^\top S(G) G_0 = R_H$. Therefore
$
Q^\top S(G) Q = -I_{n-p} \oplus R_H.
$
Taking nuclear norms gives
$
E_S(G) = \|S(G)\|_* = \|Q^\top S(G) Q\|_* = n-p + \|R_H\|_*.
$
\end{proof}

We next reduce the edge-deletion problem to a matrix problem of order $p+2$.
\begin{lemma}\label{lem3.2}
Let $G = H[n_1, \ldots, n_p]$ and let $e = uv \in E(G)$ with $u \in X_a, v \in X_b$. Suppose $n_a, n_b \geq 2$. Define $\alpha = 1 - \frac{1}{n_a}, \beta = 1 - \frac{1}{n_b}$, and $s = \frac{\varepsilon_a}{\sqrt{n_a}}, t = \frac{\varepsilon_b}{\sqrt{n_b}}$, where $\varepsilon_i$ denotes the $i$-th standard basis vector of $\mathbb{R}^p$. Then $E_S(G - e) - E_S(G) = \|M\|_* - \|M_0\|_*$, where
$$M_0 = \begin{pmatrix} -1 & 0 & 0 \\ 0 & -1 & 0 \\ 0 & 0 & R_H \end{pmatrix} \quad \text{and} \quad M = \begin{pmatrix}
-1 & 2\sqrt{\alpha\beta} & 2\sqrt{\alpha}t^\top \\
2\sqrt{\alpha\beta} & -1 & 2\sqrt{\beta}s^\top \\
2\sqrt{\alpha}t & 2\sqrt{\beta}s & R_H + 2(st^\top + ts^\top)
\end{pmatrix}.$$
Here $M_0$ and $M$ are $(p+2)\times(p+2)$ matrices.
\end{lemma}
\begin{proof}
By Lemma \ref{lem3.1}, there is an orthogonal matrix $Q$ such that $Q^\top S(G)Q = -I_{n-p} \oplus R_H$. Deleting the edge $uv$ changes the two Seidel matrix entries corresponding to $(u,v)$ and $(v,u)$ from $-1$ to $1$. Hence
$$
S(G-e) = S(G) + 2(e_u e_v^\top + e_v e_u^\top),
$$
where $e_u, e_v \in \mathbb{R}^n$ are the standard basis vectors corresponding to vertices $u$ and $v$, respectively. Let $\xi = Q^\top e_u$, $\eta = Q^\top e_v$. Write
$$
\xi = \begin{pmatrix} c \\ s \end{pmatrix}, \qquad \eta = \begin{pmatrix} d \\ t \end{pmatrix},
$$
where $c, d \in \mathbb{R}^{n-p}$ and $s, t \in \mathbb{R}^p$. Since $u \in X_a$ and $v \in X_b$, we have
$$
s = G_0^\top e_u = \frac{1}{\sqrt{n_a}} \varepsilon_a, \qquad t = G_0^\top e_v = \frac{1}{\sqrt{n_b}} \varepsilon_b.
$$
Because $Q$ is orthogonal, $\|\xi\| = \|Q^\top e_u\| = \|e_u\| = 1$, and similarly $\|\eta\| = 1$. Thus
$$
\|c\|^2 = 1 - \frac{1}{n_a} = \alpha, \qquad \|d\|^2 = 1 - \frac{1}{n_b} = \beta.
$$
Since $c = F^\top e_u$ and $d = F^\top e_v$ with $u \in X_a$, $v \in X_b$ and $a \neq b$, we have $c^\top d = 0$.

We now construct an orthogonal matrix $\widetilde{W}$ and apply one more orthogonal similarity transformation to $Q^\top S(G)Q$ and $Q^\top S(G-e)Q$. Since $n_a,n_b\ge 2$, we have $\alpha,\beta>0$, so $c$ and $d$ are nonzero. Moreover, $n-p\ge 2$ (as $n_a,n_b\ge 2$ and $a\ne b$). Hence we may choose an orthonormal basis $u_1,\ldots,u_{n-p}$ of $\mathbb{R}^{n-p}$ with
\[
u_{n-p-1}=\frac{c}{\sqrt{\alpha}},\qquad u_{n-p}=\frac{d}{\sqrt{\beta}},
\]
and let $W=(u_1,\ldots,u_{n-p})$ be the corresponding orthogonal matrix. Set
$$
\widetilde{W} = \begin{pmatrix} W & 0 \\ 0 & I_p \end{pmatrix}.
$$
Then $\widetilde{W}$ is orthogonal. Since $Q^\top S(G)Q = \begin{pmatrix} -I_{n-p} & 0 \\ 0 & R_H \end{pmatrix}$, we have
$$
(Q\widetilde{W})^\top S(G)(Q\widetilde{W}) = \widetilde{W}^\top \begin{pmatrix} -I_{n-p} & 0 \\ 0 & R_H \end{pmatrix} \widetilde{W} = \begin{pmatrix} -I_{n-p} & 0 \\ 0 & R_H \end{pmatrix}.
$$
Recall that $Q^\top S(G-e)Q = Q^\top S(G)Q + 2(\xi\eta^\top + \eta\xi^\top)$. Under the above orthogonal transformation,
$$
\widetilde{W}^\top \xi = \widetilde{W}^\top \begin{pmatrix} c \\ s \end{pmatrix} = \begin{pmatrix} W^\top c \\ s \end{pmatrix} = \bigl(0, \ldots, 0, \sqrt{\alpha}, 0, s^\top\bigr)^\top,
$$
and
$$
\widetilde{W}^\top \eta = \widetilde{W}^\top \begin{pmatrix} d \\ t \end{pmatrix} = \begin{pmatrix} W^\top d \\ t \end{pmatrix} = \bigl(0, \ldots, 0, 0, \sqrt{\beta}, t^\top\bigr)^\top.
$$
Therefore
$$
\widetilde{W}^\top Q^\top S(G-e)Q \widetilde{W} = \widetilde{W}^\top Q^\top S(G)Q \widetilde{W} + 2\bigl((\widetilde{W}^\top\xi)(\widetilde{W}^\top\eta)^\top + (\widetilde{W}^\top\eta)(\widetilde{W}^\top\xi)^\top\bigr).
$$
Since the first $n-p-2$ entries of both $\widetilde{W}^\top\xi$ and $\widetilde{W}^\top\eta$ are zero, the perturbation affects only the last $p+2$ coordinates. Hence
$$
\widetilde{W}^\top Q^\top S(G-e)Q \widetilde{W} = -I_{n-p-2} \oplus M,
$$
where
$$
M = \begin{pmatrix}
-1 & 2\sqrt{\alpha\beta} & 2\sqrt{\alpha}\,t^\top \\
2\sqrt{\alpha\beta} & -1 & 2\sqrt{\beta}\,s^\top \\
2\sqrt{\alpha}\,t & 2\sqrt{\beta}\,s & R_H + 2(st^\top + ts^\top)
\end{pmatrix}.
$$
Thus $E_S(G) = n-p-2 + \|M_0\|_*$ and $E_S(G-e) = n-p-2 + \|M\|_*$.
Consequently, $E_S(G-e) - E_S(G) = \|M\|_* - \|M_0\|_*.$
\end{proof}

We conclude this section with the following structural criterion.
\begin{theorem}\label{thm3.3}
Let $G=H[n_1,\ldots,n_p]$, and let $e=uv\in E(G)$, where
$u\in X_a$, $v\in X_b$, $a\neq b$, and $ab\in E(H)$.
Suppose that one of the following conditions holds:
\begin{enumerate}
    \item[{\em (i)}] $n_a\ge 4$ and $n_b\ge 4$;
    \item[{\em(ii)}] one of $n_a,n_b$ is equal to $3$ and the other is at least $6$;
    \item[{\em(iii)}] one of $n_a,n_b$ is equal to $2$ and the other is at least $15$.
\end{enumerate}
Then $E_S(G-e)>E_S(G)$.
\end{theorem}

\begin{proof}
By Lemma \ref{lem3.2}, $E_S(G-e)-E_S(G)=\|M\|_*-\|M_0\|_*$. Put
\[
B=\begin{pmatrix}-1&2\sqrt{\alpha\beta}\\2\sqrt{\alpha\beta}&-1\end{pmatrix},
\qquad
C=R_H+2(st^\top+ts^\top),
\qquad
X=\begin{pmatrix}2\sqrt{\alpha}\,t^\top\\2\sqrt{\beta}\,s^\top\end{pmatrix}.
\]
Then $M=\begin{pmatrix}B&X\\X^\top&C\end{pmatrix}$. By Lemma \ref{lem2.4},
$\|M\|_*\ge \|B\|_*+\|C\|_*$. Since $\|M_0\|_*=2+\|R_H\|_*$, we obtain
\[
E_S(G-e)-E_S(G)\ge (\|B\|_*-2)+(\|C\|_*-\|R_H\|_*).
\]
Under each of {\rm (i)}--{\rm (iii)}, we have
\[
\alpha\beta
=\left(1-\frac1{n_a}\right)\left(1-\frac1{n_b}\right)
\ge\frac7{15}>\frac14.
\]
Indeed, the respective lower bounds in the three cases are
$\frac34\cdot\frac34$, $\frac23\cdot\frac56$, and
$\frac12\cdot\frac{14}{15}$. Hence $2\sqrt{\alpha\beta}>1$. The
eigenvalues of $B$ are $-1+2\sqrt{\alpha\beta}$ and
$-1-2\sqrt{\alpha\beta}$, and thus
\[
\|B\|_*-2=4\sqrt{\alpha\beta}-2.
\]
Moreover, by Lemma \ref{lem2.2},
\[
\|C\|_*-\|R_H\|_*
\ge-\|2(st^\top+ts^\top)\|_*.
\]
Since $s=\frac{1}{\sqrt{n_a}}\varepsilon_a$ and
$t=\frac{1}{\sqrt{n_b}}\varepsilon_b$ are orthogonal,
$2(st^\top+ts^\top)$ has two nonzero eigenvalues
$\pm\frac{2}{\sqrt{n_an_b}}$. Consequently,
\[
E_S(G-e)-E_S(G)
\ge
4\sqrt{\left(1-\frac1{n_a}\right)
\left(1-\frac1{n_b}\right)}
-2-\frac4{\sqrt{n_an_b}}.
\]
Let
\[
f(x,y)
=
4\sqrt{\left(1-\frac1x\right)\left(1-\frac1y\right)}
-2-\frac4{\sqrt{xy}}.
\]
Then $E_S(G-e)-E_S(G)\ge f(n_a,n_b)$. For fixed $y>1$, both
$\sqrt{1-\frac1x}$ and $-\frac1{\sqrt{x}}$ are strictly increasing
in $x>1$. Hence $f(x,y)$ is strictly increasing in $x$. By symmetry,
it is also strictly increasing in $y$.
Suppose first that condition {\rm (iii)} holds. Up to interchanging
$n_a$ and $n_b$, we have $(n_a,n_b)=(2,m)$ for some $m\ge15$. Hence
\[
f(n_a,n_b)\ge f(2,15)
=4\sqrt{\frac12\cdot\frac{14}{15}}-2-\frac4{\sqrt{30}}
=\frac{4(\sqrt{14}-1)}{\sqrt{30}}-2>0.
\]
Thus $E_S(G-e)>E_S(G)$.
Suppose next that condition {\rm (ii)} holds. Up to interchanging
$n_a$ and $n_b$, we have $(n_a,n_b)=(3,m)$ for some $m\ge6$. Hence
\[
f(n_a,n_b)\ge f(3,6)
=4\sqrt{\frac23\cdot\frac56}-2-\frac4{\sqrt{18}}
=\frac{4\sqrt5-2\sqrt2-6}{3}>0.
\]
Thus $E_S(G-e)>E_S(G)$.
Finally, suppose that condition {\rm (i)} holds and
$(n_a,n_b)\neq(4,4)$. Then at least one of $n_a,n_b$ is at least $5$,
and hence
\[
f(n_a,n_b)\ge f(4,5)
=4\sqrt{\frac34\cdot\frac45}-2-\frac4{\sqrt{20}}>0.
\]
Therefore $E_S(G-e)>E_S(G)$ in all these cases.

It remains to consider $n_a=n_b=4$. Here
$\alpha=\beta=\frac34$, $s=\frac12\varepsilon_a$, and
$t=\frac12\varepsilon_b$. Thus
\[
B=\begin{pmatrix}-1&\frac32\\ \frac32&-1\end{pmatrix},
\qquad
X=\begin{pmatrix}\sqrt3\,t^\top\\ \sqrt3\,s^\top\end{pmatrix},
\qquad
C=R_H+2(st^\top+ts^\top).
\]
The eigenvalues of $B$ are $\frac12$ and $-\frac52$, and
$u=\frac1{\sqrt2}(1,1)^\top$ is a unit eigenvector corresponding to
$\frac12$. Since $ab\in E(H)$, we have
$(R_H)_{aa}=(R_H)_{bb}=3$ and
$(R_H)_{ab}=(R_H)_{ba}=-4$. Also,
$2(st^\top+ts^\top)_{ab}=2(st^\top+ts^\top)_{ba}=\frac12$.
Hence $C_{aa}=C_{bb}=3$ and $C_{ab}=C_{ba}=-\frac72$, so
\[
(s+t)^\top C(s+t)
=\frac14(C_{aa}+C_{bb}+C_{ab}+C_{ba})
=-\frac14<0.
\]
Expand $s+t$ in an orthonormal eigenbasis of $C$, say
$s+t=\sum_i\rho_i v_i$ with $Cv_i=\lambda_i(C)v_i$. Since
$\sum_i\lambda_i(C)\rho_i^2<0$, there exists an index $i$ such that
$\lambda_i(C)<0$ and $\rho_i\neq0$. Let $v=v_i$. Then
$(s+t)^\top v = \rho_i \neq0$, and hence
\[
u^\top Xv=\sqrt{\frac32}\,(s+t)^\top v\neq0.
\]
Thus the conditions for the strict part of Lemma \ref{lem2.4} are
satisfied, yielding $\|M\|_*>\|B\|_*+\|C\|_*$. Consequently,
\[
E_S(G-e)-E_S(G)>
(\|B\|_*-2)+(\|C\|_*-\|R_H\|_*).
\]
Since $\|B\|_*-2=1$ and
$\|C\|_*-\|R_H\|_*\ge-\|2(st^\top+ts^\top)\|_*=-1$, we obtain
$E_S(G-e)-E_S(G)>1-1=0$. This completes the proof.
\end{proof}

\section{A reduced-order spectral criterion for complete multipartite graphs}
In this section, we turn to complete multipartite graphs. Let $X_1,\ldots,X_r$ be the partite sets of $K_{p_1,\ldots,p_r}$, where $|X_i|=p_i$ for each $i$. We first record two immediate consequences of Theorem \ref{thm3.3}.

\begin{corollary}\label{coro4.1}
Let $G=K_{p_1,\ldots,p_r}$ with $r\geq3$, and let $e=uv\in E(G)$,
where $u\in X_a$ and $v\in X_b$, with $a\neq b$. Suppose that one of
the following conditions holds:
\begin{enumerate}
    \item[{\em (i)}] $p_a\geq4$ and $p_b\geq4$;
    \item[{\em(ii)}] one of $p_a,p_b$ is equal to $3$ and the other is at
    least $6$;
    \item[{\em(iii)}] one of $p_a,p_b$ is equal to $2$ and the other is at
    least $15$.
\end{enumerate}
Then $E_S(G-e)>E_S(G)$.
\end{corollary}

\begin{proof}
Since $G=K_{p_1,\ldots,p_r}=K_r[p_1,\ldots,p_r]$, the sets $X_1,\ldots,X_r$ are precisely the blow-up parts of $G$, with $|X_i|=p_i$ for each $i$. Thus the result follows directly from Theorem \ref{thm3.3}.
\end{proof}

\begin{corollary}\label{coro4.2}
Let $r\geq4$ and $n\geq4r$. Then for every edge $e\in E(T(n,r))$, $E_S(T(n,r)-e)>E_S(T(n,r))$.
\end{corollary}

\begin{proof}
The Tur\'an graph $T(n,r)$ is a complete $r$-partite graph whose part sizes are either $\left\lfloor \frac{n}{r} \right\rfloor$ or $\left\lceil \frac{n}{r} \right\rceil$. If $n \geq 4r$, then $\left\lfloor \frac{n}{r} \right\rfloor \geq 4$. Thus every part of $T(n,r)$ has size at least 4. The result follows from Corollary \ref{coro4.1} (i).
\end{proof}

Corollary~\ref{coro4.2} gives a complete affirmative answer to Problem \ref{prob1.2}; namely, one may take $n_0=4r$.

We now derive an exact reduced-order spectral criterion for the cases of complete multipartite graphs not covered by Corollary \ref{coro4.1}. The criterion below is valid for every edge of a complete multipartite graph; in particular, it determines the sign of the Seidel energy change in all the remaining cases of Problem \ref{prob1.1}.

Let $G = K_{p_1,\ldots,p_r}$ with $r \geq 3$ and $p_i \geq 2$, and write $n = \sum_{i=1}^r p_i$. Let $X_1, \ldots, X_r$ be the partite sets of $G$, where $|X_i|=p_i$ for each $i$. Take an edge $e = uv \in E(G)$, where $u \in X_a$ and $v \in X_b$ with $a \neq b$. After deleting $e$, we refine the partite partition as
\[
Y_1 = \{u\},\quad Y_2 = X_a \setminus \{u\},\quad
Y_3 = \{v\},\quad Y_4 = X_b \setminus \{v\},
\]
and let $Y_5, \ldots, Y_{r+2}$ be the remaining partite sets $X_i$ for $i \notin \{a,b\}$, in any order. Put $q_\mu = |Y_\mu|$ for $1 \leq \mu \leq r+2$. Then $G-e$ is an independent-set blow-up graph. By Lemma~\ref{lem3.1},
\[
S(G-e) \sim -I_{n-r-2} \oplus B_e,
\]
where $B_e = (b_{\mu\nu})$ is an $(r+2)\times(r+2)$ symmetric matrix with diagonal entries $b_{\mu\mu} = q_\mu - 1$ and off-diagonal entries $b_{\mu\nu} = \tau_{\mu\nu} \sqrt{q_\mu q_\nu}$ for $\mu \neq \nu$. Here $\tau_{\mu\nu} = \tau_{\nu\mu} = 1$ precisely for the unordered pairs
\[
\{Y_1,Y_2\},\quad \{Y_1,Y_3\},\quad \{Y_3,Y_4\},
\]
and for all other distinct pairs we have $\tau_{\mu\nu} = -1$.

For $G = K_{p_1,\ldots,p_r}$, define
\[
R = (r_{ij})_{1 \leq i,j \leq r}, \qquad
r_{ii} = p_i - 1, \qquad
r_{ij} = -\sqrt{p_i p_j} \quad (i \neq j).
\]
Since $G = K_r[p_1,\ldots,p_r]$, Lemma~\ref{lem3.1} gives $S(G) \sim -I_{n-r} \oplus R$.
Hence $\lambda_n(S(G)) = \min\{-1, \lambda_{\min}(R)\}$.

Therefore the following criterion requires only the spectra of the matrices $R$ and $B_e$, which have orders $r$ and $r+2$, respectively, rather than the spectra of the full Seidel matrices $S(G)$ and $S(G-e)$ of order $n$. The criterion is especially useful when $r$ is much smaller than $n$.

\begin{lemma}\label{lem4.3}
For $G = K_{p_1, \dots, p_r}$ with $r \geq 3$ and $p_i \geq 2$, let $e$ join $X_a$ and $X_b$, and let $R$ and $B_e$ be defined as above. Then
\[
E_S(G-e) - E_S(G) = 2\left[ s_{-}(B_e) + \min\{-1, \lambda_{\min}(R)\} - 2 \right].
\]
\end{lemma}

\begin{proof}
By Lemma \ref{lem3.1}, $E_S(G-e) = n - r - 2 + \|B_e\|_*$. Note that
\[
\operatorname{tr} B_e = \sum_{\mu=1}^{r+2} (q_\mu - 1) = \Bigl(\sum_{\mu=1}^{r+2} q_\mu\Bigr) - (r+2) = n - r - 2.
\]
By Lemma \ref{lem2.1}, for a real symmetric matrix $A$, we have $\|A\|_* = \operatorname{tr} A + 2s_-(A)$. Therefore
\[
\|B_e\|_* = n - r - 2 + 2s_-(B_e),
\]
and consequently
\[
E_S(G-e) = (n - r - 2) + \bigl(n - r - 2 + 2s_-(B_e)\bigr) = 2(n - r - 2) + 2s_-(B_e).
\]
By Theorem \ref{thm2.5}, for $G = K_{p_1,\ldots,p_r}$ we have
\[
E_S(G) = 2n - 2r - 2\lambda_n(S(G)),
\]
where $\lambda_n(S(G))= \min\{-1, \lambda_{\min}(R)\}$. Subtracting this from $E_S(G-e)$ gives the desired formula.
\end{proof}

Lemma \ref{lem4.3} immediately yields the following exact reduced-order spectral criterion.
\begin{theorem} \label{thm4.4}
Let $G = K_{p_1, \dots, p_r}$ with $r \geq 3$ and $p_i \geq 2$, let $e$ join $X_a$ and $X_b$, and let $R$ and $B_e$ be defined as above. Set
$\theta = 2 - \min\{-1, \lambda_{\min}(R)\}$. Then
\begin{align*}
E_S(G-e) > E_S(G) &\Longleftrightarrow s_{-}(B_e) > \theta, \\
E_S(G-e) < E_S(G) &\Longleftrightarrow s_{-}(B_e) < \theta, \\
E_S(G-e) = E_S(G) &\Longleftrightarrow s_{-}(B_e) = \theta.
\end{align*}
\end{theorem}
\begin{proof}
The result follows immediately from Lemma \ref{lem4.3}.
\end{proof}

The following example illustrates the use of Theorem \ref{thm4.4} and shows that, even for a fixed complete multipartite graph, the sign of the Seidel energy change may depend on the partite sets containing the deleted edge.
\begin{example}\label{ex4.5}
Let $G=K_{2,2,4}$ with partite sets $X_1,X_2,X_3$ of sizes $2$, $2$, and $4$, respectively. The matrix $R$ is
\[
R=
\begin{pmatrix}
1&-2&-2\sqrt2\\
-2&1&-2\sqrt2\\
-2\sqrt2&-2\sqrt2&3
\end{pmatrix}.
\]
Its smallest eigenvalue is $\lambda_{\min}(R)=1-2\sqrt5$, and hence $\theta=2-\lambda_{\min}(R)=1+2\sqrt5$. First, let $e_{12}$ join $X_1$ and $X_2$. Then
\[
B_{e_{12}}=
\begin{pmatrix}
0&1&1&-1&-2\\
1&0&-1&-1&-2\\
1&-1&0&1&-2\\
-1&-1&1&0&-2\\
-2&-2&-2&-2&3
\end{pmatrix}.
\]
The negative eigenvalues of $B_{e_{12}}$ are $-3$ and $-\sqrt5$. Thus $s_-(B_{e_{12}})=3+\sqrt5<1+2\sqrt5=\theta$. By Theorem \ref{thm4.4}, $E_S(G-e_{12})<E_S(G)$. Next, let $e_{13}$ join $X_1$ and $X_3$. Then
\[
B_{e_{13}}=
\begin{pmatrix}
0&1&1&-\sqrt3&-\sqrt2\\
1&0&-1&-\sqrt3&-\sqrt2\\
1&-1&0&\sqrt3&-\sqrt2\\
-\sqrt3&-\sqrt3&\sqrt3&2&-\sqrt6\\
-\sqrt2&-\sqrt2&-\sqrt2&-\sqrt6&1
\end{pmatrix}.
\]
Numerically, $s_-(B_{e_{13}})\approx5.591209>1+2\sqrt5=\theta$. By Theorem \ref{thm4.4}, $E_S(G-e_{13})>E_S(G)$. Thus, for the fixed graph $K_{2,2,4}$, deleting an edge joining the two parts of size $2$ decreases the Seidel energy, whereas deleting an edge joining a part of size $2$ to the part of size $4$ increases the Seidel energy.
\end{example}

\begin{remark}
The equality alternative in Theorem \ref{thm4.4} is included for completeness. We leave the existence and classification of equality cases for complete multipartite graphs open here.
\end{remark}

\section{Concluding remarks and open problems}
Sections 3 and 4 established structural and spectral criteria for the change of Seidel energy under edge deletion in independent-set blow-up graphs and complete multipartite graphs. It is natural to ask whether analogous criteria hold for arbitrary graphs.

\begin{problem}\label{prob5.1}
Let $G$ be an arbitrary graph and let $e\in E(G)$. Find effective structural or spectral criteria that determine the sign of $E_S(G-e)-E_S(G)$.
\end{problem}

The Seidel matrix has a natural interpretation in terms of signed graphs. Recall that a signed graph $\Sigma = (\Gamma, \sigma)$ consists of an underlying graph $\Gamma$ and a sign function $\sigma: E(\Gamma) \to \{+1, -1\}$. Its adjacency matrix is defined by $A(\Sigma)_{uv} = \sigma(uv)$ when $uv \in E(\Gamma)$, and by $A(\Sigma)_{uv} = 0$ otherwise. The adjacency energy of $\Sigma$ is $\mathcal{E}(\Sigma) = \|A(\Sigma)\|_*$. For a graph $G$, define a signed complete graph $\Sigma_G$ on $V(G)$ with signs:
\[
\sigma(uv) = \begin{cases}
-1, & uv \in E(G),\\
1, & uv \notin E(G).
\end{cases}
\]
Then $A(\Sigma_G) = S(G)$ and $\mathcal{E}(\Sigma_G) = E_S(G)$. Deleting an edge $uv$ from $G$ corresponds to changing the sign of $uv$ in $\Sigma_G$ from $-1$ to $+1$. Thus, Problem \ref{prob5.1} may be viewed as a sign-flip problem for signed complete graphs. Viewed in this way, Theorem \ref{thm3.3} and Theorem \ref{thm4.4} provide, respectively, structural and exact reduced-order spectral criteria for the sign-flip problem on the associated signed complete graphs, thereby motivating its extension to non-complete signed graphs.

\begin{problem}\label{prob5.2}
Let \(\Sigma\) be a signed graph and let \(uv\) be a negative edge. Let $\Sigma^{+uv}$ be obtained from $\Sigma$ by changing the sign of $uv$ from $-1$ to $+1$. Find effective structural or spectral criteria that determine the sign of $ \mathcal E(\Sigma^{+uv})-\mathcal E(\Sigma). $
\end{problem}


\end{document}